\documentclass{ietperso}
\usepackage{xcolor}
\usepackage{hyperref}
\hypersetup{
    colorlinks=true,
    linkcolor=blue,
    citecolor=red,
    filecolor=magenta,      
    urlcolor=cyan,
    pdftitle={Overleaf Example},
    pdfpagemode=FullScreen,
    }
\usepackage{ulem}
\newtheorem{thm}{Theorem}

\newtheorem{lem}[thm]{Lemma}
\newtheorem{prop}[thm]{Proposition}
\theoremstyle{definition}
\newtheorem{defn}[thm]{Definition}
\newtheorem{rem}[thm]{Remark}
\newtheorem{example}[thm]{Example}

\numberwithin{equation}{section}
\allowdisplaybreaks

\newcommand{\Id}{\operatorname{Id}}

\newcommand{\R}{\ensuremath{\mathbb R}}    
\newcommand{\C}{\ensuremath{\mathbb C}}    
\newcommand{\N}{\ensuremath{\mathbb N}}    

\newcommand{\<}{\langle}
\renewcommand{\>}{\rangle}

\newcommand{\calA}{\mathcal A}

\newcommand{\calH}{\mathcal H}

\newcommand{\calN}{\mathcal N}

\newcommand{\calT}{\mathcal T}         
\newcommand{\calU}{\mathcal U}

\newcommand{\calY}{\mathcal Y}

\newcommand{\eps}{\varepsilon}
\newcommand{\vphi}{\varphi}

\newcommand{\bvec}[2]
{
	\begin{bmatrix}
		#1\\
		#2
	\end{bmatrix}
}

\renewcommand{\Im}{\operatorname{Im}}
\renewcommand{\Re}{\operatorname{Re}}

\newcommand{\ran}{\operatorname{ran}}

\newcommand{\sra}{\rightarrow}

\newcommand{\Sra}{\Rightarrow}

\newcommand{\ol}{\overline}

\newcommand{\wt}{\widetilde}

\newcommand{\ptl}{\partial}
\begin{document}
\title{Turnpike property for generalized linear-quadratic problems}
\author[1]{Zhuqing Li}
\author[2,*]{Roberto Guglielmi}
\affil[1]{Department of Mechanical and Aerospace Engineering, University of California San Diego}
\affil[2]{Department of Applied Mathematics, University of Waterloo}
\affil[*]{\url{roberto.guglielmi@uwaterloo.ca}}

\maketitle

\begin{ietabstract}
We deduce a sufficient condition of the exponential (integral) turnpike property for infinite dimensional generalized linear-quadratic optimal control problems in terms of structural properties of the control system, such as exponential stabilizability and detectability. The proof relies on the analysis of the exponential convergence of solutions to the differential Riccati equations to the algebraic counterpart, and on a necessary condition for exponential stabilizability in terms of a closed range test. 
\end{ietabstract}

The authors state that they have followed the ethics and integrity policies on the website for author guidelines. In particular, the data that support the findings of this study are available from the corresponding author upon reasonable request. The authors acknowledge the support of the Natural Sciences and Engineering Research Council of Canada (NSERC), funding reference number RGPIN-2021-02632. The authors have no conflict of interests related to this publication.

\begin{ietbody}
\section{Introduction}

The turnpike phenomenon characterizes the long-term behavior of optimally controlled systems, where the optimal trajectories over an extended time horizon closely align with a predefined trajectory of the system. The term turnpike property was first coined by Samuelson~\cite{SAMUELSON}, referring to the property of an efficiently growing economy to approximate for most of the time horizon a balanced equilibrium path, as observed by von Neumann~\cite{Neumann}. Subsequently, turnpike phenomena have been extensively explored in various domains, including mathematical economics~\cite{MCKENZIE}, mathematical biology~\cite{BIO}, and chemical processes~\cite{CHE}.

Over the last decade, the mathematical community has thoroughly investigated turnpike properties in optimal control, particularly in connection with the stability of Model Predictive Control schemes~\cite{discreteexp,RHC}. These investigations cover linear and nonlinear settings, discrete-time systems~\cite{FINDIS,FINNON}, finite-dimensional continuous-time systems~\cite{FINCON,TreZuaFIN}, delay equations~\cite{DELAY}, infinite-dimensional systems with control inside the domain~\cite{INFICAR,Zua,TreZua}, optimal relaxed problems~\cite{OPREL}, and constrained problems~\cite{Boundary}. Comprehensive overviews of turnpike properties in various optimal control and variational problems can be found in the monographs~\cite{ZAS06,ZAS14,ZAS15}.

In most instances, the trajectory approximated by the optimally controlled system corresponds to the minimizer of the corresponding optimal steady state problem, representing an optimal steady state of the system. However, turnpike phenomena have been observed towards different target states, such as suitable periodic orbits~\cite{TreZua,Periodicstrict,SAMUELSONperiodic}.

Numerous methods have been developed to establish sufficient conditions for turnpike properties, broadly categorized into two techniques. The first technique leverages hyperbolicity around the steady state of the optimality system resulting from Pontryagin's maximum principle. These results are instrumental in proving the integral and exponential integral turnpike property (see Definition~\ref{turnpikedef}) of the optimal state, control, and adjoint state towards the solution of the corresponding optimal steady state problem~\cite{TreZuaFIN,Zua,TreZua}.

The second method exploits the connection between turnpike and dissipativity properties of the control system. Dissipativity, introduced by Willems~\cite{Will72a,Will72b}, characterizes the abstract energy balance of a dynamical system in terms of stored and supplied energy. Strict dissipativity and strict pre-dissipativity, useful in characterizing turnpike properties, have been introduced for this purpose~\cite{Grue13,GruM16,survey}. This approach enables the establishment of the measure turnpike property, ensuring that the Lebesgue measure of the set where the optimal trajectory deviates significantly from the turnpike reference is bounded by a constant independent of the time horizon. While the measure turnpike property is more general than the exponential integral turnpike, it does not provide the same quantitative estimates on the control and adjoint state. Further details and a comparison between measure and integral turnpike properties can be found in~\cite{MEAINT}.

The linear-quadratic (LQ) optimal control problem is particularly intriguing in the study of turnpike properties due to its specific structure, allowing for a more explicit analysis. Extensive studies of this case have been conducted for finite-dimensional LQ optimal control problems in discrete and continuous time, including the presence of state and control constraints~\cite{GruG18,GruG21}.

This paper aims to generalize the results presented in~\cite{Zua} within a different framework. Specifically, we focus on the generalized linear-quadratic optimal control problem, where the quadratic term of the running cost function need not be strictly positive, and the linear terms of the running cost function can be arbitrarily chosen. In this sense, our problem setting does not restrict to optimal control problems of tracking type. As a consequence, the existence of a minimizer of the optimal steady state problem does not directly imply the existence of an adjoint steady state as in \cite{Zua}. This issue is further discussed in Remark \ref{adjointexist} and Appendix \ref{appenB}, which provides a counterexample to the existence of an adjoint steady state. 

The general framework considered in this paper requires to rely on a suitable semigroup framework, offering a distinct perspective on the turnpike property for infinite-dimensional problems. Due to the different functional framework, several proofs from~\cite{Zua} require reworking. For instance, we develop a closed range test to ensure the existence of the adjoint steady state based on the stabilizability of the control system. This test is utilized to derive the representation formula of the explicit solution of the generalized LQ problem. To our knowledge, this formula is not known in the literature and holds intrinsic interest. Furthermore, our method leverages the analytic properties of evolution systems and a suitable class of perturbations of the optimal control problem to deduce a key estimate on the exponential convergence of the solutions to the differential Riccati equations to its algebraic counterpart. This process holds promise for generalization to the unbounded observation case, as discussed in~\cite[Section 4.3]{thesis}. We anticipate addressing this issue for admissible observation operators in future work.

The remainder of the paper is organized as follows: In Section~\ref{Probsetting}, we introduce the generalized LQ problem and the optimal steady state problem, define the exponential (integral) turnpike property in infinite dimension, and state our main result along with some existing results on Riccati equations and perturbations of generators. In Section~\ref{infisetting} we first develop essential lemmas on the exponential convergence of solutions to the differential Riccati equations to its algebraic counterpart and the closed range test. Then we prove our main result based on several estimates, which enable us to treat the uncontrolled evolution system like an exponentially stable semigroup, and on the explicit solution of the optimal control. In Section~\ref{appli} we conclude with the application of our main result to two examples. We start by setting some notations.

{\bf Notations}
Throughout the paper, we use $\mathcal{H},\,\mathcal{U}$ and $\mathcal{Y}$ to denote the state, input and output space, respectively, and assume they are all complex Hilbert spaces. We use the same notation $\<\cdot,\cdot\> $ and $\|\cdot\|$ for the inner product and induced norm in all spaces $\calH$,  $\calU$ and $\calY$. We denote by $L(X,Y)$ the space of linear bounded operators between the Hilbert spaces $X$ and $Y$. Assume $A:D(A)\to \calH$ is the generator of a strongly continuous semigroup $\calT = \left(\calT_t\right)_{t\ge 0}$ in $\calH$ and $B\in L(\calU, \calH). $ We assume $C\in L(\calH,\calY),\,K\in L(\calU,\calH),\,z\in \mathcal{H}$ and $v\in \mathcal{U}$. Here \(K\) is further assumed to be coercive, i.e., there exists $M>0$ such that $\<K^*Ku,u\>\geq M\|u\|^2$ for all $u\in \calU$. {We use $S^*$ to denote the adjoint operator associated to operator $S$.}

{\bf Functional framework} 
We denote by $\calH_{1},\,\calH_{1}^d$ the Hilbert spaces defined on the sets $D(A)$ and $D(A^*)$, respectively, endowed with the inner products
\begin{align*}
&\<x,y\>_{\calH_{1}}:=\<(sI-A)x,(sI-A)y\>,\qquad\forall x,y\in D(A)\\
&\<x,y\>_{\calH_{1}^d}:=\<(sI-A^*)x,(sI-A^*)y\>,\qquad\forall x,y\in D(A^*)
\end{align*}
for some $s>w_0(\calT)$, where $w_0(\calT)$ is the growth bound of $\calT$. The induced norms of the two inner products are known to be equivalent to the graph norm of $A$, $A^*$. See, e.g.,~\cite[Section~2.10]{TusWei}.

We denote by $\calH_{-1},\,\calH_{-1}^d$ the Hilbert spaces being the space completions of $\calH$ with respect to the inner products
\begin{align*}
&\<x,y\>_{\calH_{-1}}:=\<(sI-A)^{-1}x,(sI-A)^{-1}y\>,\qquad\forall x,y\in \calH\\
&\<x,y\>_{\calH_{-1}^d}:=\<(sI-A^*)^{-1}x,(sI-A^*)^{-1}y\>,\qquad\forall x,y\in \calH.
\end{align*}

It is well-known that $\calH_{-1},\,\calH_{-1}^d$ are isomorphic to the dual spaces of $\calH_{1}^d,\,\calH_{1}$. For the sake of simplicity, we will identify $\calH_{-1},\,\calH_{-1}^d$ with the dual spaces of $\calH_{1}^d,\,\calH_{1}$ by not explicitly writing out the isomorphism between these spaces. So, we may define
\begin{align*}
\<x,\phi\>_{\calH_{1},\calH_{-1}^d}&:=\phi(x),\qquad\forall x\in\calH_{1},\, \phi\in \calH_{-1}^d\\
\<x,\phi\>_{\calH_{1}^d,\calH_{-1}}&:=\phi(x),\qquad\forall x\in\calH_{1}^d,\, \phi\in \calH_{-1}\\
\<\phi,x\>_{\calH_{-1}^d,\calH_{1}}&:=\ol{\<x,\phi\>_{\calH_{1},\calH_{-1}^d}},\qquad\forall x\in\calH_{1},\, \phi\in \calH_{-1}^d\\
\<\phi,x\>_{\calH_{-1},\calH_{1}^d}&:=\ol{\<x,\phi\>_{\calH_{1}^d,\calH_{-1}}},\qquad\forall x\in\calH_{1}^d,\, \phi\in \calH_{-1}.
\end{align*}
{Notice that the above notations are linear in the first argument and antilinear in the second argument, similarly to the inner product on a Hilbert space.} Moreover, if $\phi\in \calH$, we have
\begin{align*}
\<x,\phi\>_{\calH_{1},\calH_{-1}^d}=\<x,\phi\>,\,\forall x\in\calH_{1},\\
\<x,\phi\>_{\calH_{1}^d,\calH_{-1}}=\<x,\phi\>,\,\forall x\in\calH_{1}.
\end{align*}
We refer to \cite[Sections 2.9 and 2.10]{TusWei} for more details on the embeddings $\calH_{1}\subset \calH \subset \calH_{-1}^d$.

\section{Problem setting and results}\label{Probsetting}
\subsection{Mathematical setting}
{\bf The OCP.} Given a time horizon $T > 0$, we consider the infinite dimensional generalized linear-quadratic Optimal Control Problem {(OCP)$_T$}
\begin{equation}\label{OCP}
\min_{u\in L^2([0,T],\mathcal{U})}\int_0^T\ell(x(t),u(t))\,dt\quad \mathrm{s.t.}\quad \dot x=Ax+Bu,\quad x(0) = x_0\in \calH\, ,
\end{equation}
with running cost function $\ell : \mathcal{H}\times \mathcal{U}\to \R$ defined by
\begin{equation}\label{cost}
\ell(x,u) := \|Cx\|^2 + \|Ku\|^2 + 2\Re\< z,x\> + 2\Re\< v,u\>.
\end{equation}
The solution of system
\begin{equation}\label{sy}
\addtolength{\jot}{5pt}
\left\{
\begin{alignedat}{2}
&\frac{dx}{dt}=Ax+Bu\,\,\,\,\text{in}\,\,[0,T],\\
&x(0)=x_0.
\end{alignedat}
\right.
\end{equation}
is considered in the mild sense, that is,
$$
x(t)=\calT_tx_0+\int_0^t\calT_{t-s}Bu(s)ds, \qquad \forall t\ge 0\;.
$$
{We denote by $x^*_T(\cdot,x_0)$ and $u^*_T(\cdot,x_0)$, or simply by $x^*$ and $u^*$ when $x_0$ and $T$ are clear from the context, the optimal trajectory and optimal control of \eqref{OCP} corresponding to the initial condition $x_0\in \calH$ and time horizon $T>0$.} Notice that, due to the strict convexity of $\ell$ with respect to~$u$ and the linearity of \eqref{sy}, the optimal control $u^*$ of \eqref{OCP} exists and is unique. {See, e.g., \cite[Theorem 1.1]{Lions}.}
Now, we are well-prepared to define the {integral and exponential integral} turnpike property at a steady state. We recall that a controlled equilibrium of system~\eqref{sy} is a pair $(x_e,u_e)\in\mathcal{H}\times\mathcal{U}$ such that $Ax_e+Bu_e=0$.
\begin{defn}\label{turnpikedef}
We say that the optimal control problem \eqref{OCP} has the integral turnpike property at some controlled equilibrium $(x_e,u_e)$ if, for any bounded neighborhood $\calN$ of $x_e$ and $\eps>0$, {the optimal trajectory $x^*_T(\cdot,x_0)$ and optimal control $u^*_T(\cdot,x_0)$ of \eqref{OCP} with initial condition $x_0\in\calN$ satisfy
\begin{align*}
\frac{1}{(b-a)T}\int_{aT}^{bT}x^*_T(t,x_0)dt\to x_e\;,\quad\frac{1}{(b-a)T}\int_{aT}^{bT}u^*_T(t,x_0)dt\to u_e
\end{align*}
as $T\to \infty$ for every $a,b\in[0,1]$, $a\neq b$.} Additionally, if there exist positive constants $M, w>0$ independent of $T>0$ and $x_0\in \calN$ such that
$$
\|x^*_T(t,x_0)-x_e\|+\|u^*_T(t,x_0)-u_e\|\leq M(e^{-wt}+e^{-w(T-t)}),
$$
we say that the optimal control problem~\eqref{OCP} shows the exponential (integral) turnpike property at $(x_e,u_e)$.
\end{defn}
\begin{rem}
Since $B$ is a bounded operator, if $(x_e,u_e)$ is a controlled equilibrium, we have that $Ax_e=-Bu_e\in\calH$. This implies that $x_e$ must be an element of $D(A)$.
\end{rem}

To study the turnpike property, it is useful to analyse the optimal steady state corresponding to the running cost $\ell$. The optimal steady state problem is defined as:
\begin{align}\label{OSSP}
\min_{x \in D(A),\,u\in\calU} \ell(x,u)\qquad\text{s.t. }Ax+Bu = 0.
\end{align}

We will also make use of several structural-theoretical properties of the control system under consideration, which we introduce in the following.

We say that the pair $(A,C)$ is exponentially detectable if there exists some $L\in L(\calY,\calH)$ such that $A+LC$ generates an exponentially stable semigroup. Similarly, $(A,B)$ is called exponentially stabilizable if there exists $F\in L(\calH,\calU)$ such that $A+BF$ is exponentially stable. 

The operator $[A\;B]:D(A)\times \calU\to \calH$ is defined by
$$
[A\;B]\bvec{x}{u}:=Ax+Bu.
$$
Observe that, since $A$ and $B$ are both closed operators, $\ker[A\;B]$ is closed in $\calH\times\calU$.

We denote by $\bvec{A^*}{B^*}$ the adjoint operator of $[A\;B]$. It is easy to check that the domain of $\bvec{A^*}{B^*}$ is $D(A^*)$ and that we have
\begin{align*}
\bvec{A^*}{B^*}w=\bvec{A^*w}{B^*w}{\in \calH\times\calU}\;,\quad \forall w\in D(A^*).
\end{align*}

\begin{rem}\label{exanduniofossp}
{With reference to} \cite[Lemma~10 and Theorem~17]{LFM}, if we assume $(A,C)$ is exponentially detectable and \(K\) is coercive, then the optimal steady state problem~\eqref{OSSP} has a unique minimizer $(x_e,u_e)$, which is characterized by
\begin{equation}\label{oss}
\mathbb P_{\ker[A\;B]}\bvec{z+C^*Cx_e}{v+K^*Ku_e} = 0,
\end{equation}
where $\mathbb P_{\ker[A\;B]}$ is the projection operator from $\calH\times \calU$ to its closed subspace $\ker[A\;B]$.
\end{rem}

In the following, we will assume that the optimal steady state is at $(x_e,u_e) = (0,0)$, thanks to the next result.
\begin{lem}\label{changestate}
Assume that $(A,C)$ is exponentially detectable. Then the solution of the optimal steady state problem \eqref{OSSP} is at $(x_e,u_e)$ if and only if {the solution of the modified optimal steady state problem
\begin{align}\label{tlstate}
\min_{x \in D(A),\,u\in\calU} \ell(x,u)\qquad\text{s.t. }Ax+Bu = 0,
\end{align}
is at $(0,0)$, where $\wt{\ell}$ is defined by
\begin{align}\label{tlcost}
\wt{\ell}(x,u) := &\|Cx\|^2 + \|Ku\|^2 + 2\Re\langle z+C^*Cx_e,x\rangle + 2\Re\langle v+K^*Ku_e,u\rangle\; .
\end{align}
In this case, the OCP \eqref{OCP}-\eqref{cost} satisfies the turnpike property at $(x_e,u_e)$ if and only if the modified OCP
\begin{align}\label{MOCP}
\min_{u\in L^2([0,T],\mathcal{U})}\int_0^T\wt{\ell}(x(t),u(t))\,dt\,\quad \mathrm{s.t. }\ \dot{x}=Ax+Bu,\quad x(0) = x_0\in\calH
\end{align}
satisfies the turnpike property at $(0,0)$.}
\end{lem}
\begin{proof}
For any $x\in\calH$ and $u\in\calU$, set $\wt{x}:=x-x_e$ and $\wt{u}:=u-u_e$. 
If $(x_e,u_e)$ is the minimizer of the optimal steady state problem \eqref{OSSP}, we easily see that $(0,0)$ is the minimizer of
\begin{align}
\min_{\wt{x} \in D(A),\,\wt{u}\in\calU} \ell(\wt{x}+x_e,\wt{u}+u_e)\qquad\text{s.t. }A\wt{x}+B\wt{u} = 0
\end{align}
and vice versa. {It is not hard to see that $\ell(\wt{x}+x_e,\wt{u}+u_e)$ only differs from $\wt{\ell}(\wt{x},\wt{u})$ by a constant term. So, $(0,0)$ is the minimizer of problem \eqref{tlstate}.}

Now, let $(x^*,u^*)$ be the optimal pair for OCP \eqref{OCP}. Observe that
$$
\frac{d\wt{x}}{dt}=\frac{d(x-x_e)}{dt}=Ax+Bu=A\wt{x}+B\wt{u}.
$$
{So, $(\wt{x}^*,\wt{u}^*):=(x^*-x_e,u^*-u_e)$ is the optimal pair of problem
\begin{align*}
\min_{\wt{u}\in L^2([0,T],\mathcal{U})}\int_0^T\ell(\wt{x}(t)+x_e,\wt{u}(t)+u_e)\,dt\,\,\,\mathrm{s.t.}\, \dot{\wt{x}}=A\wt{x}+B\wt{u},\quad \wt{x}(0) = \wt{x}_0:=x_0-x_e.
\end{align*}
As mentioned before, $\ell(\wt{x}+x_e,\wt{u}+u_e)$ only differs from $\wt{\ell}(\wt{x},\wt{u})$ by a constant term. So, we have that $(x^*-x_e,u^*-u_e)$ is the optimal pair for modified OCP \eqref{MOCP}} and vice versa. Therefore, we deduce that OCP \eqref{OCP} satisfies turnpike property at $(x_e,u_e)$ if and only if the modified OCP \eqref{OCP}-\eqref{tlcost} satisfies turnpike property at $(0,0)$.
\end{proof}

\subsection{Riccati equations and perturbation of generators}
In this section, we introduce some well-known properties of differential and algebraic Riccati equations. {For more details on this section, we refer to \cite[Part IV, Section 1]{Bensou} and \cite[Part V, Section 1]{Bensou} and the references therein.}

Let $\Sigma(\calH)$ denote the Hilbert space of self-adjoint operators endowed with their operator norm. Let $C_s([0,T],\Sigma(\calH))$ (resp. $C_s^1([0,T],\Sigma(\calH))$) denote the linear vector space (endowed with the strong operator topology) of all $P:[0,T]\to \Sigma(\calH)$ such that $P(\cdot)x$ is a continuous (resp.~continuously differentiable) function on $[0,T]$, for every $x\in \calH$.

The differential Riccati equation is given by
\begin{equation}\label{Riccati}
\addtolength{\jot}{5pt}
\left\{
\begin{alignedat}{2}
&\frac{dP}{dt}-A^*P-PA+PB(K^*K)^{-1}B^*P-C^*C=0\\
&P(0)=P_0\in \Sigma(\calH).
\end{alignedat}
\right.
\end{equation}
\begin{defn}
We call $P\in C_s([0,T],\Sigma(\calH))$ a (weak) solution of the differential Riccati equation~\eqref{Riccati} if $P(0)=P_0$ and, for any $x,y\in D(A)$, $\<P(t)x,y\>$ is differentiable on $[0,T]$ and verifies, for each $t\in[0,T]$,
\begin{equation}\label{DRE}
\frac{d\<P(t)x,y\>}{dt}-\<P(t)x,Ay\>-\<P(t)Ax,y\>+\<(K^*K)^{-1}B^*P(t)x,B^*P(t)y\>-\<Cx,Cy\>=0.
\end{equation}
\end{defn}

If our assumptions about $A,B,C,K$ are verified and $P_0$ is non-negative, then there exists a unique non-negative solution to equation \eqref{Riccati}. See, e.g., \cite[Part IV, Theorem~2.1]{Bensou}. It's well-known that the optimal control for OCP \eqref{OCP} with cost function $\ell(x,u) := \|Cx\|^2 + \|Ku\|^2$ is given by the feedback law 
$$
u^*(t)=-(K^*K)^{-1}B^*P(T-t)x^*(t)\; ,
$$ 
where $P$ is the solution of \eqref{Riccati} with $P(0)=0$. 
Sometimes we want $P$ to have stronger regularity by imposing some restrictions on $P_0$. To define the strict solutions of \eqref{Riccati}, for any $F\in \Sigma(\calH)$, we introduce the bilinear mapping
$$
\vphi_F(x,y):=\<Fx,Ay\>+\<Ax,Fy\>,\quad \forall x,y\in D(A).
$$
Let
$
D(\calA):=\{F\in \Sigma(\calH)\,|\,\vphi_F\,\text{has a continuous extension to}\,\calH\times\calH\}.
$
For any $F\in D(\calA)$, since $\vphi_F$ has a continuous extension in $\calH\times\calH$, there exists a unique bounded linear operator $\calA(F):\calH\to\calH$ satisfying 
$$
\<\calA(F)x,y\>=\<Fx,Ay\>+\<Ax,Fy\>,\quad \forall x,y\in D(A),\,F\in D(\calA).
$$
\begin{defn}
We say that $P\in C_s([0,T],\Sigma(\calH))$ is a strict solution of the differential Riccati equation \eqref{Riccati} if $P$ satisfies
\begin{enumerate}
\item[(a)] $P\in C_s^1([0,T],\Sigma(\calH)),$
\item[(b)] $P(t)\in D(\calA)$ for all $t\geq 0,$
\item[(c)] $\calA(P)\in C_s([0,T];\Sigma(\calH))$ and $P{'}=\calA(P)-PB(K^*K)^{-1}B^*P+C^*C,\, P(0)=P_0.$
\end{enumerate}
\end{defn}

\begin{rem}\label{whystrictsolution}
If $P$ is a strict solution of \eqref{Riccati}, then $P$ is also a mild solution. {If our assumptions about $A,B,C,K$ are verified, $P_0$ is non-negative and $P_0\in D(\calA)$}, then there exists a unique strict solution to the differential Riccati equation \eqref{Riccati}. See, e.g., \cite[Part~IV, Theorem 3.2]{Bensou}. Notice, in particular, that $0\in D(\calA)$. 
\end{rem}

We now introduce the algebraic Riccati equation, which is given by
\begin{equation}\label{ARiccati}
A^*P+PA-PB(K^*K)^{-1}B^*P+C^*C=0.
\end{equation}
\begin{defn}
We call $P\in \Sigma(\calH)$ a solution of the algebraic Riccati equation \eqref{ARiccati} if for any $x,y\in D(A)$, we have
$$
\<Px,Ay\>+\<PAx,y\>-\<(K^*K)^{-1}B^*Px,B^*Py\>+\<Cx,Cy\>=0.
$$
\end{defn}
{If our assumptions about $A,B,C,K$ are verified and the pair $(A,B)$ is exponentially stabilizable}, then there exists a minimal non-negative solution to the algebraic Riccati equation. See, e.g., \cite[Part V, Proposition 2.3]{Bensou}. We denote such minimal solution by $P_{\min}$, which is also the unique non-negative solution of \eqref{ARiccati} if we additionally assume that $(A,C)$ is exponentially detectable. It is well known that if $(A,B)$ is exponentially stabilizable, then the optimal control problem on infinite time horizon
\begin{equation}\label{infOCP}
\min_{u\in L^2([0,\infty),\mathcal{U})}\int_0^\infty \|Cx(t)\|^2 + \|Ku(t)\|^2\,dt\,\,\,\mathrm{s.t.}\, \dot x=Ax+Bu,\quad x(0) = x_0
\end{equation}
is well-posed, and the optimal control $u^*$ is given by the feedback law
$$
u^*(t)=-(K^*K)^{-1}B^*P_{\min}x^*(t).
$$
Moreover, the semigroup generated by $A-BB^*P_{\min}$ is exponentially stable.

We conclude the section by recalling two useful results from~\cite{Bensou}. The first one concerns the convergence of $P$ to $P_{\min}$. In the next section we will generalize the exponential convergence 
in the infinite dimensional setting.
\begin{prop}\label{findimP}
Suppose $P$ is the solution of \eqref{Riccati} with $P(0)=0$ and $P_{\min}$ is the minimal solution of \eqref{ARiccati}, then we have $P(t)\sra P_{\min}$ in strong operator topology. If $\mathcal{H},\,\mathcal{U}$ and $\mathcal{Y}$ are all Euclidean spaces, then $P(t)\sra P_{\min}$ in norm topology and the convergence rate is exponential, i.e., there exists positive constants $M,\eta>0$ such that 
$$\|P(t)-P_{\min}\|\leq Me^{-\eta t}.$$
\end{prop}
We refer to \cite[Part V, Proposition 4.2]{Bensou} for a proof of the convergence in strong operator topology. On the other hand, the exponential convergence rate in finite dimensions is a well-known result. In fact, since in finite dimensional spaces the convergence in strong topology is equivalent to the convergence in operator topology, we have that $P(t)\to P_{\min}$ in matrix norm. The exponential rate is then a consequence of the local stability result~\cite[Part V, Section 1, Proposition 4.1]{Bensou}. 
Finally, we recall from \cite[Proposition 3.5]{Bensou} a perturbation result that will be useful in the proof of our main theorem.

\begin{prop}\label{evolution}
Assume that $F\in C_s([0,T],L(\calH))$ and $f\in L^2((0,T),\calH)$. Then the problem
\begin{equation}\label{evopro}
\begin{split}
\addtolength{\jot}{5pt}
\left\{
\begin{alignedat}{2}
&\frac{dx(t)}{dt}=Ax(t)+F(t)x(t)+f(t),\,\,\,\,t\in[0,T],\\
&x(0)=x_0\in \calH
\end{alignedat}
\right.
\end{split}
\end{equation}
admits a unique solution in $L^2((0,T),\calH)$, that is, $x\in H^1((0,T),\calH_{-1})\cap C([0,T],\calH)$ and 
$$
\frac{dx(t)}{dt}=Ax(t)+Fx(t)+f(t) \,\,\,\text{in}\,\,\,\,\calH_{-1},\,\,\,\,\,\forall t\in[0,T].
$$

If, in addition, $F\in C_s^1([0,T],L(\calH))$, $x_0\in D(A)$ and $f\in H^1((0,T),\calH)$, then the solution of problem \eqref{evopro} is a strict one, i.e., $x\in C^1([0,T],\calH)\cap C([0,T],\calH_1)$.
\end{prop}

\subsection{Results}
Our main result is motivated by the following theorem in finite dimension~\cite[Theorem 8.1]{GruG21}.
\begin{thm}\label{mainfin}
Suppose $\mathcal{H}=\mathbb R^n,\,\mathcal{U}=\R^m$ and $\mathcal{Y}=\R^k$ with $m,n,k\in \mathbb N^+$. The following statements are equivalent:\\
\qquad (a) $(A,C)$ is detectable and $(A,B)$ is stabilizable.\\
\qquad (b) The OCP \eqref{OCP} has the (measure) turnpike property at some controlled equilibrium.
\end{thm}
It is natural to ask if the above theorem can be extended to the infinite dimensional setting. The main result of this paper, stated in the following theorem, provides a generalization of the implication $(a)\Sra(b)$ in infinite dimension. It must be noted that the turnpike property proved in \cite{GruG21} is the so-called measure turnpike property, while in this paper we deal with the exponential (integral) turnpike property, which is slightly stronger than the former property. In fact, our main result is more related to \cite{Zua} both in terms of techniques and in terms of the notion of turnpike. However, we refer to Theorem \ref{mainfin} from \cite{GruG21} to emphasise that even in finite dimension condition $(a)$ is necessary for measure turnpike property, thus also for exponential integral turnpike property.

We now state the main result of the paper. Its proof will be divided into several steps, developed in the following sections.

\begin{thm}\label{maininf}
Assume that $(A,B)$ is exponentially stabilizable and $(A,C)$ is exponentially detectable, then OCP \eqref{OCP} has the (exponential) turnpike property at the optimal steady state.
\end{thm}

\section{Methods}\label{infisetting}

\subsection{Exponential convergence of $P(t)$ to $P_{\min}$} 

The proof of our main result crucially relies on the exponential convergence of the solution to the differential Riccati equation $P(t)$ to the solution of the algebraic Riccati equation $P_{\min}$, that is, to the generalization of Proposition~\ref{findimP} in the infinite dimensional setting. This turns out to be a consequence of both the stabilizability and detectability of the system. This section is devoted to the proof of this result, preceded by two preliminary lemmas. The same sort of estimate was used in a critical way in~\cite[Lemma 3.5 and Lemma 3.9]{Zua}, where it was proved by combining a careful analysis of the optimality system and classical parabolic equation techniques from~\cite{Lions}. In this work, instead, in order to cope with the generalized optimal control problem~\eqref{OCP}-\eqref{sy}, we must extend the estimate in a more abstract framework, by means of semigroup methods.

\begin{lem}\label{detect}
Assume $(A,C)$ is exponentially detectable, then there exists a constant $M>0$ such that for every $u\in L^2([0,T],\calU), x_0\in \calH$, we have
\begin{equation}\label{deteest}
\|x(T)\|^2\leq M \left(\int_0^T\|Cx(t)\|^2+\|u(t)\|^2dt+\|x_0\|^2\right),
\end{equation}
where $x$ is the solution of system \eqref{sy}.
\end{lem}

\begin{proof}
Since $(A,C)$ is exponentially detectable, there exists some $F\in L(\calY,\calH)$ such that $A+FC$ is exponentially stable.
Let $\phi_0\in D(A^*)$, and let $\phi$ be the solution to
\begin{equation*}
\addtolength{\jot}{5pt}
\left\{
\begin{alignedat}{2}
&\frac{d\phi}{dt}=(A^*+C^*F^*)\phi\,\,\,\,\text{in}\,\,[0,T],\\
&\phi(0)=\phi_0.
\end{alignedat}
\right.
\end{equation*}
Then, owing to the exponential stability of $A^*+C^*F^*$, there exists some $M>0$ such that for arbitrary $T> 0$, we have {$\|\phi\|_{L^\infty([0,T],\calH)}\leq M\|\phi_0\|$ and $\|\phi\|_{L^2([0,T],\calH)}\leq M\|\phi_0\|$.}
Notice that
\begin{align*}
\frac{d\left\<x(t),\phi(T-t)\right\>}{dt}&=\left\<x(t),-(A^*+C^*F^*)\phi(T-t)\right\>+\left\<Ax(t)+Bu(t),\phi(T-t)\right\>_{\calH_{-1},\calH_{1}^d}\\
&=-\<Cx(t),F^*\phi(T-t)\>+\<Bu(t),\phi(T-t)\>.
\end{align*}
So, applying Hölder's inequality, we have that
\begin{align*}
\<x(T),\phi_0\>=& \<x_0,\phi(T)\>+\int_0^T-\left\<Cx(t),F^*\phi(T-t)\right\>+\left\<Bu(t),\phi(T-t)\right\>dt\\
\leq&
\|x_0\|\left\|\phi(T)\right\|+\left\|Cx\right\|_{L^2([0,T],\calY)}\|F^*\phi\|_{L^2([0,T],\calY)}+\|Bu\|_{L^2([0,T],\calH)}\|\phi\|_{L^2([0,T],\calH)}\\
\leq& M_1\|x_0\|\|\phi_0\|+M_2\|Cx\|_{L^2([0,T],\calY)}\|\phi_0\|+M_3\|u\|_{L^2([0,T],\calU)}\|\phi_0\|\\
\leq& \sqrt{M}\|\phi_0\|\left(\|Cx\|^2_{L^2([0,T],\calY)}+\|u\|^2_{L^2([0,T],\calU)}+\|x_0\|^2\right)^{\frac{1}{2}}
\end{align*}
for some constants $M_1, M_2, M_3$ and $M>0$.

Now, by letting $\phi_0\to x(T)$ in $\calH$, we get the estimate~\eqref{deteest}.
\end{proof}
\begin{lem}\label{stab}
Assume $(A,B)$ is exponentially stabilizable, then there exists $M>0$ such that, for every $f\in L^2([0,T],\calY)$ and $p_0\in \calH$, we have
\begin{equation}\label{stabest}
\|p(T)\|^2\leq M \left(\int_0^T\|B^*p(t)\|^2+\|f(t)\|^2dt+\|p_0\|^2\right)\; ,
\end{equation}
where $p$ is the solution of
\begin{equation*}
\addtolength{\jot}{5pt}
\left\{
\begin{alignedat}{2}
&\frac{dp}{dt}=A^*p+C^*f\,\,\,\,\text{in}\,\,[0,T],\\
&p(0)=p_0.
\end{alignedat}
\right.
\end{equation*}
\end{lem}
This is just the dual version of Lemma \ref{detect}, so we skip the proof.

\begin{lem}\label{expcon}
(Exponential convergence rate of $P$) If $(A,B)$ is exponentially stabilizable and $(A,C)$ is exponentially detectable, then there exist some positive numbers $M,\beta$ such that
\begin{equation}\label{expP}
\|P_{\min}-P(t)\|\leq Me^{-\beta t},\quad \forall t\geq0.
\end{equation}
\end{lem}
\begin{proof}
Fix some $T>0$. We denote by $x$ the solution of
\begin{equation}\label{x}
\addtolength{\jot}{5pt}
\left\{
\begin{alignedat}{2}
&\frac{dx}{dt}=(A-BB^*P_{\min})x\,\,\,\,\text{in}\,\,[0,T],\\
&x(0)=x_0\in \calH.
\end{alignedat}
\right.
\end{equation}
Define $p(t):=P_{\min}x(T-t)$ in $[0,T]$. We claim that $p$ coincides with the solution of
\begin{equation}\label{p}
\addtolength{\jot}{5pt}
\left\{
\begin{alignedat}{2}
&\frac{dp(t)}{dt}=A^*p(t)+C^*Cx(T-t)\,\,\,\,\text{in}\,\,[0,T],\\
&p(0)=P_{\min}x(T).
\end{alignedat}
\right.
\end{equation}
In fact, suppose $x_0\in D(A)$, then for any $y\in D(A)$, we have
\begin{align}\label{expconest1}
\begin{split}
\frac{d\<p(t),y\>}{dt}=&-\<P_{\min}Ax(T-t),y\>+\<P_{\min}BB^*P_{\min}x(T-t),y\>\\
=&\<A^*P_{\min}x(T-t)+C^*Cx(T-t),y\>_{\calH_{-1}^d,\calH_1}.
\end{split}
\end{align}
So, we have $\frac{dp(t)}{dt}=A^*p(t)+C^*Cx(T-t)$ in $\calH_{-1}^d$. 

Now, for any $x_0\in \calH$, we let $(z_n)_{n\in \N}\subset D(A)$ be a sequence such that $\lim_{n\to \infty}z_n=x_0$ in $\calH$. We use $x_n$ and $p_n$ to denote the solutions of \eqref{x} and \eqref{p} with $x_0$ replaced by $z_n$. It is clear that $x_n\to x$ in the $L^2$ norm and $P_{\min}x_n(T)\to P_{\min}x(T)$ as $n\to\infty$. So, we have that $p_n$ converges to the solution of \eqref{p} pointwisely. It is also clear that $p_n(t)=P_{\min}x_n(T-t)\to P_{\min}x(T-t)$ pointwisely. So, we must have that $p(t)$ is equal to the solution of \eqref{p}.

Let $P$ be the solution of the differential Riccati equation \eqref{Riccati} with $P_0=0$. Consider the following evolution problem
\begin{equation}\label{tx}
\addtolength{\jot}{5pt}
\left\{
\begin{alignedat}{2}
&\frac{d\wt{x}(t)}{dt}=(A-BB^*P(T-t))\wt{x}(t)\,\,\,\,\text{in}\,\,[0,T],\\
&\wt{x}(0)=x_0.
\end{alignedat}
\right.
\end{equation}
{Notice that $P(T-\cdot)\in C_s([0,T],L(\calH))$. Thanks to Proposition~\ref{evolution}, there exists a unique solution $\wt{x}\in H^1([0,T],\calH^{-1})\cap C([0,T],\calH)$ of problem \eqref{tx}.} 
Define $\wt{p}(t):=P(t)\wt{x}(T-t)$ in $[0,T]$. Similarly as before, we claim 
that $\wt{p}$ coincides with the solution of
\begin{equation}\label{tp}
\addtolength{\jot}{5pt}
\left\{
\begin{alignedat}{2}
&\frac{d\wt{p}(t)}{dt}=A^*\wt{p}(t)+C^*C\wt{x}(T-t)\,\,\,\,\text{in}\,\,[0,T],\\
&\wt{p}(0)=0.
\end{alignedat}
\right.
\end{equation}
Indeed, if $x_0\in D(A)$, then by Remark \ref{whystrictsolution} and Proposition \ref{evolution} we have 
\begin{equation}\label{wtxreg}
\wt{x}\in C^1([0,T],\calH)\cap C([0,T],D(A)).
\end{equation}
{Let $t\in[0,T]$ and $h\in\R$ be sufficiently small.} Observe that
\begin{align*}
\frac{\wt{p}(t+h)-\wt{p}(t)}{h}=P(t+h)\left(\frac{\wt{x}\left(T-(t+h)\right)-\wt{x}(T-t)}{h}+(A-BB^*P(T-t))\wt{x}(T-t)\right)\\
-P(t+h)(A-BB^*P(T-t))\wt{x}(T-t)+\frac{P(t+h)-P(t)}{h}\wt{x}(T-t).
\end{align*}
By the uniform boundedness principle applied to $P$, \eqref{tx}, \eqref{wtxreg} and Remark \ref{whystrictsolution}, taking $h\to0$, we obtain
\begin{align}\label{tpest}
\begin{split}
\frac{d\wt{p}(t)}{dt}&=-P(t)(A-BB^*P(t))\wt{x}(T-t)\\
&\qquad\quad+(\calA(P(t))-P(t)BB^*P(t)+C^*C)\wt{x}(T-t)\\
&=(\calA(P(t))-P(t)A)\wt{x}(T-t)+C^*C\wt{x}(T-t)\\
&=A^*P(t)\wt{x}(T-t)+C^*C\wt{x}(T-t)\\
&=A^*\wt{p}(t)+C^*C\wt{x}(T-t)
\end{split}
\end{align}
in $\calH^{-1}_d$. Since $\wt{p}(0)=0$, this proves our claim for the case $x_0\in D(A)$.

Now for any $x_0\in\calH$, we let $(\wt{z}_n)_{n\in \N}\subset D(A)$ be a sequence such that $\lim_{n\to \infty}\wt{z}_n=x_0$ in $\calH$. We use $\wt{x}_n$ and $\wt{p}_n$ to denote the solutions of \eqref{tx} and \eqref{tp} with $x_0$ replaced by $\wt{z}_n$. By~\cite[Proposition 3.6]{Bensou}, $\wt{x}_n\to \wt{x}$ uniformly (thus also in $L^2$ norm), so $\wt{p}_n$ converges to the solution of \eqref{tp} pointwisely. It's clear $\wt{p}_n$ also converges to $\wt{p}$ pointwisely. So, the solution of \eqref{tp} coincides with $\wt{p}$.

Now we claim that following inequality holds for any $x_0\in\calH$:
\begin{equation}\label{dual}
\int_0^T\|B^*(p(t)-\wt{p}(t))\|^2+\|C(x(t)-\wt{x}(t))\|^2dt\leq\|p(0)\|\|\wt{x}(T)-x(T)\|\; ,
\end{equation}
where $x,p,\wt{x}$ and $\wt{p}$ are the solutions of \eqref{x}, \eqref{p}, \eqref{tx} and \eqref{tp}, respectively. 
To prove this, we first assume that $x_0\in D(A)$. Since $x,\,\wt{x}\in C([0,T],D(A))$, we have
\begin{align*}
&\frac{d\<\wt{p}(t)-p(t),\wt{x}(T-t)-x(T-t)\>}{dt}\\
&\qquad=-\<p^*(t)-p(t),A(\wt{x}(T-t)-x(T-t))\>+\<\wt{p}(t)-p(t),BB^*(\wt{p}(t)-p(t))\>\\
&\qquad\qquad\,\,+\<A^*(\wt{p}(t)-p(t)),\wt{x}(T-t)-x(T-t)\>_{\calH_{-1}^d,\calH_1}\\
&\qquad\qquad\,\,+\<C^*C(\wt{x}(T-t)-x(T-t)),\wt{x}(T-t)-x(T-t)\>\\
&\qquad=\|C(x(T-t)-\wt{x}(T-t))\|^2+\|B^*(p(t)-\wt{p}(t))\|^2.
\end{align*}
So, we have that
\begin{align}\label{ana1}
\begin{split}
&\int_0^T\|B^*(p(t)-\wt{p}(t))\|^2+\|C(x(t)-\wt{x}(t))\|^2dt\\
&\qquad=\<\wt{p}(T)-p(T),\wt{x}(0)-x(0)\>-\<\wt{p}(0)-p(0),\wt{x}(T)-x(T)\>\\
&\qquad=\<p(0),\wt{x}(T)-x(T)\>\\
&\qquad\leq\|p(0)\|\|\wt{x}(T)-x(T)\|\; .
\end{split}
\end{align}
Since for any $x_0\in \calH$, we can find a sequence $(z_n)_{n\in \N}\subset D(A)$ so that $\lim_{n\to \infty}z_n=x_0$ in $\calH$, and the resulting trajectories $x_n,p_n,\wt{x}_n$ and $\wt{p}_n$ will converge to $x,p,\wt{x}$ and $\wt{p}$ both pointwisely and in $L^2$ norm, \eqref{dual} still holds for all $x_0\in \calH$.

Notice that
$$
\frac{d(\wt{x}(t)-x(t))}{dt}=A(\wt{x}(t)-x(t))-BB^*(\wt{p}(T-t)-p(T-t)).
$$
By Lemma \ref{detect}, we have that, for some constant $M>0$,
\begin{equation}\label{diffxest0}
\|\wt{x}(T)-x(T)\|^2\leq M\int_0^T\|B^*(p(t)-\wt{p}(t))\|^2+\|C(x(t)-\wt{x}(t))\|^2dt.
\end{equation}
Applying \eqref{diffxest0} to \eqref{ana1}, we obtain
\begin{equation}\label{diffxest}
\int_0^T\|B^*(p(t)-\wt{p}(t))\|^2+\|C(x(t)-\wt{x}(t))\|^2dt\leq M\|p(0)\|^2.
\end{equation}
Since
$$
\frac{d(\wt{p}(t)-p(t))}{dt}=A^*(\wt{p}(t)-p(t))-C^*C(\wt{x}(T-t)-x(T-t)),
$$
by Lemma \ref{stab} there exists some $M>0$ such that
\begin{equation}\label{diffpest}
\|\wt{p}(T)-p(T)\|^2\leq M\left(\int_0^T\|B^*(p(t)-\wt{p}(t))\|^2+\|C(x(t)-\wt{x}(t))\|^2dt+\|p(0)\|^2\right).
\end{equation}
Since $(A,C)$ is exponentially detectable, $A-BB^*P_{\min}$ generates an exponentially stable semigroup. So, there exists some constants $M,\beta>0$ such that
\begin{equation}\label{pexp}
\|p(0)\|=\|P_{\min}x(T)\|\leq Me^{-\beta T}\|x_0\|.
\end{equation}
Substituting \eqref{diffxest} and \eqref{pexp} into \eqref{diffpest}, we get, for some $M>0$,
\begin{equation*}
\|\wt{p}(T)-p(T)\|^2\leq Me^{-2\beta T}\|x_0\|^2.
\end{equation*}
Finally, since $\|\wt{p}(T)-p(T)\|=\|\left(P_{\min}-P(T)\right)x_0\|$, we get \eqref{expP}.
\end{proof}

\subsection{Closed range test}\label{closedrange}

In this subsection, we introduce a closed range test which enables us to find the explicit solution of the optimal control problem~\eqref{OCP}.

\begin{lem}\label{test}
If $(A,B)$ is exponentially stabilizable, then $\ran[A\;B]$ is closed.
\end{lem}
\begin{proof}
Suppose $\bvec{z}{v}\in \ol{\ran}\bvec{A^*}{B^*}$, then there exists a sequence $(w_n)_{n\in \N}\subset D(A^*)$ such that $\bvec{A^*}{B^*}w_n\to \bvec{z}{v}$ in $\calH\times \calU$. Set $z_n:=A^*w_n$ and $v_n:=B^*w_n$. 

For any $x_0\in \calH$ and $u\in L^2([0,T],\calU)$, let $x$ denote the solution of 
$$
\dot x=Ax+Bu,\quad x(0) = x_0\in \calH.
$$
We then compute that
\begin{align}\label{appen}
\begin{split}
&\int_0^T2\Re\<z_n,x(t)\>+2\Re\<v_n,u(t)\>dt =\int_0^T2\Re\<w_n,Ax(t)+Bu(t)\>_{\calH_{1}^d,\calH_{-1}}dt\\
&\qquad=2\Re\<w_n,x(T)-x_0\>.
\end{split}
\end{align}
Clearly, as $n\to\infty$, we have that
\begin{equation*}
\int_0^T2\Re\<z_n,x(t)\>+2\Re\<v_n,u(t)\>dt\to\int_0^T2\Re\<z,x(t)\>+2\Re\<v,u(t)\>dt.
\end{equation*}
So, putting together the two previous relations, we deduce that $2\Re\<w_n,x(T)-x_0\>$ will necessarily converge to some real number as $n\to\infty$.

We claim that $x(T)-x_0$ can actually attain any value in $\calH$. In fact, since $(A,B)$ is exponentially stabilizable, there exists some $K\in L(\calH,\calU)$ such that $A+BK$ generates an exponentially stable semigroup $\calT^{cl}$. So, for sufficiently large $T>0$, $\ran(\calT^{cl}_T-I)=\calH$, since $\calT^{cl}_T-I$ is a invertible operator whose inverse is given by the Neumann series $-\sum_{i=0}^{\infty}(\calT^{cl}_T)^i$. Then, for any $q\in\calH$, there exists $x_0\in \calH$ such that $x(T)-x_0=\calT^{cl}_Tx_0-x_0=q$, where $u$ is given by $u(t):=K\calT_t^{cl}x_0$. 

{We thus have proved that, for any $q\in\calH$, $2\Re\<q,w_n\>$ will converge to some real number as $n\to\infty$. Let us denote such value by $f(q)$.} The function $f$ is linear (over the field $\R$) and real. In order to apply the Riesz's representation theorem, we need to ensure that $f$ is a bounded linear functional (over the field $\R$). It turns out that the boundedness of $f$ is also a consequence of the exponential stabilizability of $(A,B)$. Indeed, fix some $T>0$ so that $\calT^{cl}_T-I$ is invertible. Then there exists some constants $M_1,M_2>0$ such that, for any $n\in\N$, we have
\begin{align*}
&2\Re\<w_n,\calT_T^{cl}x_0-x_0\> { = \int_0^T2\Re\<w_n,(A+BK)\calT_t^{cl}x_0\>_{\calH_{1}^d,\calH_{-1}}dt}\\
&\qquad=\int_0^T 2\Re\<z_n,\calT_t^{cl}x_0\>+2\Re\<v_n,K\calT_t^{cl}x_0\>dt\leq\int_0^TM_1\|\calT_t^{cl}x_0\|dt\\
&\qquad\leq M_2\|x_0\| \leq M_2\|(\calT_T^{cl}-I)^{-1}\|\|\calT_T^{cl}x_0-x_0\|.
\end{align*}
Since $\ran(\calT^{cl}_T-I)=\calH$, we have
$$
f(q)=\lim_{n\to\infty}2\Re\<w_n,q\>\leq C_2\|(\calT_T^{cl}-I)^{-1}\|\|q\|\,,\quad \forall q\in\calH\; .
$$
So $f$ is bounded. By Theorem \ref{Riesz} in Appendix \ref{appenA}, we know that $f(q)=2\Re\<q,w\>$ for some $w\in\calH$. By noticing that $\lim_{n\to\infty}2\Re\<w_n,q\>=2\Re\<w,q\>$ and $\Re\<w,iq\>=\Im\<w,q\>$, we know that $\lim_{n\to\infty}\<w_n,q\>=\<w,q\>$ for any $q\in\calH$. So we have that $w_n\to w$ weakly in $\calH$. The above equality also implies that $\lim_{n\to\infty}\<A^*w_n,q\>=\<w,Aq\>=\<A^*w,q\>_{\calH_{-1}^d,\calH_{1}}$ for any $q\in D(A)$, and thus $A^*w_n\to A^*w$ weakly in $\calH_{-1}^{d}$ as $n\to \infty$. 
By the definition of $(w_n)_{n\in\N}$, we have that $A^*w_n\to z$ (strongly) as $n\to \infty$ in $\calH$, thus also in $\calH_{-1}^{d}$. So, $z=A^*w\in\calH$ and $w\in D(A^*)$. We now prove that $B^*w=v$. As mentioned earlier, $w_n\to w$ weakly in $\calH$, so $B^*w_n\to B^*w$ weakly in $\calU$. By definition, we have that $B^*w_n\to v$ (strongly) in $\calU$, thus we deduce that $B^*w=v$ and $\bvec{z}{v}\in \ran\bvec{A^*}{B^*}$. So, $\ran\bvec{A^*}{B^*}$ is closed. {By the closed graph theorem, $\ran[A\;B]$ is closed.}
\end{proof}

\subsection{Properties of the evolution operator}\label{twoestimate}

In this subsection, we will make some simplifications to OCP \eqref{OCP} and prove two crucial estimates. From now on, we assume that the pair $(A,C)$ is exponentially detectable and the pair $(A,B)$ is exponentially stabilizable.

Without loss of generality, we may also assume that $K=I$. Otherwise, we may define a new inner product $\<\cdot,\cdot\>_{\text{new}}$ on $\calU$ by
$$
\<u_1,u_2\>_{\text{new}}=\<(K^*K)^{\frac{1}{2}}u_1,(K^*K)^{\frac{1}{2}}u_2\>,\,\,\forall u_1,u_2\in \calU.
$$
We now endow $\calU$ with the new inner product $\<\cdot,\cdot\>_{\text{new}}$. Since the norm induced by this inner product is equivalent to the standard norm in $\calU$, we still have that $B\in L(\calU,\calH)$. The running cost $\ell$ shall now be recast as
\begin{equation*}
\ell(x,u) := \|Cx\|^2 + \|u\|_{\text{new}}^2 + 2\Re\<z,x\> + 2\Re\<(K^*K)^{-1}v,u\>_{\text{new}}
\end{equation*}
where $\|\cdot\|_{\text{new}}$ is the norm on $\calU$ induced by $\<\cdot,\cdot\>_{\text{new}}$.

By Lemma \ref{exanduniofossp}, there exists a unique minimizer $(x_e,u_e)$ of the optimal steady state problem \eqref{OSSP}. By \eqref{oss} and Lemma \ref{test}, we have that
$$
\bvec{z+C^*Cx_e}{v+u_e}\in \ker[A\;B]^{\perp}=\ol{\ran}\bvec{A^*}{B^*}=\ran\bvec{A^*}{B^*}.
$$
So, there exists some $w\in D(A^*)$ such that
\begin{align}\label{charw}
\bvec{A^*w}{B^*w}=\bvec{z+C^*Cx_e}{v+u_e}.
\end{align}

\begin{rem}\label{adjointexist}
In the works based on the study of the hyperbolicity of the optimality system, e.g., \cite{Zua,TreZua,TreZuaFIN}, the vector $w$ actually corresponds to the optimal adjoint steady state. In our setting, however, we have to use different techniques to ensure the existence of $w$. In particular, unlike \cite{TreZua, TreZuaFIN}, the existence of a solution to the optimal steady state problem does not imply the existence of such a vector $w$ (we provide a counterexample in Appendix \ref{appenB}). Instead, the existence of $w$ is now guaranteed by the closed range condition of Lemma~\ref{test}. Besides, in \cite{TreZuaFIN,TreZua} the vector $w$ is understood as the Lagrange multiplier (see Appendix \ref{appenB}), it is characterized by
\begin{align*}
\bvec{A^*w}{BB^*w}=\bvec{z+C^*Cx_e}{Bv+BK^*Ku_e}.
\end{align*}
Compared to \eqref{charw}, in our setting we do not have $B$ applied to the second line. This is essential for the {identity~\eqref{cov}, which transforms the linear terms of the running cost $\ell$ into a terminal cost $2\Re\<w,x(T)\>$, as introduced below.}
\end{rem}

At this moment, we assume that the optimal steady state $(x_e,u_e)=(0,0)$. Let $P$ be the solution of the differential Riccati equation \eqref{Riccati} with $P(0)=0$, then we have
\begin{align}\label{Pxx}
\begin{split}
&\frac{d\<P(T-t)x(t),x(t)\>}{dt}=-\<P^{'}(T-t)x(t),x(t)\>\\
&\qquad+2\Re\<P(T-t)x(t),Ax(t)\>+2\Re\<B^*P(T-t)x(t),u(t)\>.
\end{split}
\end{align}
We now apply a change of variable $\wt{u}(t):=u(t)+BB^*P(T-t)x(t)$. By combining~\eqref{DRE} and \eqref{Pxx}, we get the identity
\begin{align}\label{cov}
\begin{split}
\int_0^T&\ell(x(t),u(t))dt=\int_0^T\|Cx(t)\|^2+\|u(t)\|^2+2\Re\<z,x(t)\>+2\Re\<v,u(t)\>dt\\
&=\int_0^T2\Re\<B^*P(T-t)x(t),u(t)\>+\|B^*P(T-t)x(t)\|^2+\|u(t)\|^2\\
&\qquad\qquad-\frac{d\<P(T-t)x(t),x(t)\>}{dt}+2\Re\<w,Ax(t)+Bu(t)\>dt\\
&=\int_0^T\|\wt{u}(t)\|^2dt+\<P(T)x_0,x_0\>+2\Re\<w,x(T)\>-2\Re\<w,x_0\>.
\end{split}
\end{align}
We notice that $x$ and $\wt{u}$ are related by the following relation
\begin{equation}\label{refo}
\dot x(t)=(A-BB^*P(T-t))x(t)+B\wt{u}(t).
\end{equation}
Since $P$ is a continuous function with respect to the strong topology, by Proposition~\ref{evolution} we deduce that problem \eqref{refo} is well posed. 
Now, for any $\tau\in [0,T)$, consider the evolution problem
\begin{equation}\label{Xeq}
\dot x(t)=(A-BB^*P(T-t))x(t),\quad x(\tau)=x_0\in \calH,\quad t\in [\tau,T]\; ,
\end{equation}
and its corresponding evolution operator $U_T(t,\tau):\calH\to\calH$, that is, $U_T(t,\tau)$ is defined by
$$
U_T(t,\tau)x_0:=x(t)
$$
with $x(t)$ being the solution of \eqref{Xeq} at $t\in [\tau, T]$. Then, for each $x_0\in\calH$, $U_T(t,\tau)x_0$ satisfies
$$
\frac{dU_T(t,\tau)x_0}{dt}=(A-BB^*P(T-t))U_T(t,\tau)x_0,\quad U_T(\tau,\tau)x_0=x_0.
$$
Next, we prove two crucial properties of the evolution operator $U_T$.

\begin{lem}\label{exprate} (Exponential convergence of $U_T$) There exist some positive constants $M,k$ such that
\begin{equation}\label{expU_T}
\|U_T(t,\tau)\|\leq Me^{-k(t-\tau)}
\end{equation}
holds for all possible values of $\tau, t, T$ with $0\leq\tau\leq t\leq T$.
\end{lem}
\begin{proof}
Suppose $\tau, t, T\in\R$ satisfy $0\leq\tau\leq t\leq T$. By Proposition \ref{findimP}, there exist $M,\mu>0$ such that $\|P(t)-P_{\min}\|\leq Me^{-\mu t}$. 
For any $x_0\in\calH$, let $x$ be the solution to \eqref{Xeq} and let $y(t):=e^{\lambda (t-\tau)}x(t)$
with some sufficiently small $\lambda>0$ so that $A-BB^*P_{\min}+\lambda I$ still generates an exponentially stable semigroup, denoted $\calT^{cl}$. A straightforward computation leads to
\begin{align}\label{estony}
\begin{split}
\dot y(t)&=(A-BB^*P_{\min}+\lambda I)y(t)+(BB^*P_{\min}-BB^*P(T-t))y(t)\\
y(t)&=\calT^{cl}_{t-\tau}y(\tau)+\int_{\tau}^t \calT^{cl}_{t-s}(BB^*P_{\min}-BB^*P(T-s))y(s)ds
\end{split}
\end{align}
for all $t\in[\tau,T]$. So, there exist constants $M_1$ and $M_2>0$ such that,
\begin{align}\label{Gronmat}
\begin{split}
\|y(t)\|&\leq M_1\|y(\tau)\|+\int_{\tau}^t M_2e^{-\mu (T-s)}\|y(s)\|ds\\
&\leq M_1\|x_0\|+M_2e^{-\mu (T-t)}\int_{\tau}^t \|y(s)\|ds.\\
\end{split}
\end{align}
Now fix some constant $S>0$ such that $M_2e^{-\mu S}< \lambda$.

We first discuss the case $\tau\geq T-S$. Referring to \eqref{Gronmat}, we can further obtain
\begin{equation}\label{Gronmat2}
\|y(t)\|\leq M_1\|x_0\|+M_2\int_{\tau}^t \|y(s)\|ds.
\end{equation}
Since $t-\tau\leq S$, applying Grönwall's inequality to \eqref{Gronmat2} we get
\begin{equation}\label{unicase2}
\|y(t)\|\leq M_1\|x_0\|e^{M_2(t-\tau)}\leq M_1e^{M_2S}\|x_0\|.
\end{equation}
So, we deduce that
\begin{equation}\label{form1}
\|U_T(t,\tau)x_0\|=\|x(t)\|\leq \|y(t)\|\leq M_1e^{(M_2+1)S}e^{-(t-\tau)}\|x_0\|,
\end{equation}
which proves condition~\eqref{expU_T} with suitable coefficients. 

We now consider the case $\tau < T - S$. If $t\in [\tau,T-S]$, from~\eqref{Gronmat} we can further obtain
\begin{equation}\label{Gronmat3}
\|y(t)\|\leq M_1\|x_0\|+M_2e^{-\mu S}\int_{\tau}^t \|y(s)\|ds.
\end{equation}
Applying Grönwall's inequality to \eqref{Gronmat3}, we have
\begin{align*}
\|y(t)\|\leq M_1\|x_0\|&e^{M_2e^{-\mu S}(t-\tau)}.
\end{align*}
We thus get that
\begin{align}\label{form2}
\|x(t)\|=e^{-\lambda (t-\tau)}\|y(t)\|\leq M_1\|x_0\|e^{(M_2e^{-\mu S}-\lambda)(t-\tau)}.
\end{align}
Since $M_2e^{-\mu S}-\lambda<0$, we conclude that~\eqref{expU_T} is satisfied with suitable coefficients also in this case. It remains to prove~\eqref{expU_T} in the case $\tau < T - S$ and $t\in (T-S,T]$. From the definition of $U_T$, we know that
\begin{align}\label{decomposition}
U_T(t,\tau)=U_{T-\tau}(t-\tau,0)=U_{S}(S-T+t,0)U_{T-\tau}(T-S-\tau,0).
\end{align}
Notice that $U_{S}(S-T+t,0)$ satisfies the estimate \eqref{form1}, so we have
\begin{equation}\label{Gron4}
\|U_{S}(S-T+t,0)\|\leq M_1e^{(M_2+1)S}e^{-(S-T+t)}\|x_0\|\leq M_1e^{(M_2+1)S}\|x_0\|.
\end{equation}
It is also clear that $U_{T-\tau}(T-S-\tau,0)$ satisfies the estimate \eqref{form2}, so we have
\begin{align}\label{Gron5}
\|U_{T-\tau}(T-S-\tau,0)\|\leq M_1e^{(M_2e^{-\mu S}-\lambda)(T-S-\tau)}\|x_0\|\leq M_1e^{(M_2e^{-\mu S}-\lambda)(t-S-\tau)}\|x_0\|.
\end{align}
Combining \eqref{decomposition}-\eqref{Gron5}, we deduce that~\eqref{expU_T} is satisfied also in this case with suitable coefficients $M$ and $k$.

By choosing the largest $M$ and the smallest $k$ among all the 3 cases, we conclude that~\eqref{expU_T} holds for all possible values of $\tau, t, T$ such that $0\leq\tau\leq t\leq T$.
\end{proof}

\begin{lem}\label{unibound}
The operators $\Phi^T_t: L^2([0,t],\mathcal{U})\sra \calH$ defined by 
$$
\Phi^T_t(\wt{u}):=\int_0^tU_T(t,s)B\wt{u}(s)ds
$$ 
are uniformly bounded in norm for all possible values of $t,T$ with $0\leq t\leq T$.
\end{lem}

\begin{proof}
By Lemma \ref{exprate} and Hölder's inequality, we have that
\begin{align*}
\|\Phi^T_t(\wt{u})\|&\leq\int_0^t\left\|U_T(t,s)\right\|\left\|B\wt{u}(s)\right\|ds\\
&\leq \left(\int_0^t M^2e^{-2k(t-s)} ds\right)^{\frac{1}{2}}\left(\int_0^t \|B\|^2\|\wt{u}(s)\|^2 ds\right)^{\frac{1}{2}}\\
&\leq M\|B\|\left(\frac{1-e^{-2kt}}{2k}\right)^{\frac{1}{2}}\|\wt{u}\| \leq M\|B\|\left(\frac{1}{2k}\right)^{\frac{1}{2}}\|\wt{u}\|
\end{align*}
where $M$ and $k$ are the coefficients given in Lemma \ref{exprate}.
So, we conclude that $\Phi^T_t$ are uniformly bounded in norm for all possible values of $t$ and $T$ with $0\leq t\leq T$.
\end{proof}

\subsection{Explicit solutions of the optimal control}
In this subsection, we give an explicit representation of the optimal control.
\begin{lem}\label{explicitoc} Let $\hat{U}_T(t,\tau)$ be defined as the evolution operator of the problem
\begin{align}\label{optconsy}
\dot z(t)=(A^*-P(t)BB^*)z(t)\,,\quad z(\tau)=x_0\in \calH\,, \quad t\in[\tau,T]
\end{align}
where $P$ is the solution of \eqref{Riccati} with $P(0)=0$. Then the optimal control of problem \eqref{OCP} is given in a feedback form by
\begin{align}\label{feedbacku}
u^*(t)=-(K^*K)^{-1}B^*P(T-t)(x^*(t)-x_e)-(K^*K)^{-1}B^*\hat{U}_{T}(T-t,0)w+u_e\, ,
\end{align}
where $w$ is the vector given in \eqref{charw}.
\end{lem}
\begin{proof}
We first assume the optimal steady state $(x_e,u_e)=(0,0)$ and $K=I$. By Remark \ref{whystrictsolution} and Proposition~\ref{evolution}, since $w\in D(A^*)$, we have
$$
\hat{U}_T(\cdot,0)w\in C([0,T],D(A^*))\cap C^1([0,T],\calH).
$$
From~\eqref{optconsy} we deduce that
\begin{align*}
\frac{d\hat{U}_T(T-t,0)w}{dt}=-(A^*-P(T-t)BB^*){U}_T(T-t,0)w.
\end{align*}
Let $x$ be any solution of problem \eqref{sy} with initial condition $x_0\in\calH$, then
\begin{align}\label{expest2}
\begin{split}
&2\Re\<x(T),w\>-2\Re\left\<x_0,\hat{U}_{T}(T,0)w\right\>\\
&\quad= 2\int_0^T\frac{d\Re\left\<x(t),\hat{U}_{T}(T-t,0)w\right\>}{dt}dt\\
&\quad=\int_0^T2\Re\left\<Ax(t)+Bu(t),\hat{U}_{T}(T-t,0)w\right\>_{H_{-1},H_{1}^d}\\
&\qquad\qquad+2\Re\left\<x(t),-(A^*-P(T-t)BB^*)\hat{U}_{T}(T-t,0)w\right\>dt\\
&\quad=\int_0^T2\Re\left\<u(t),B^*\hat{U}_{T}(T-t,0)w\right\>\\
&\qquad\qquad+2\Re\left\<B^*P(T-t)x(t),B^*\hat{U}_{T}(T-t,0)w\right\>dt.
\end{split}
\end{align}
Combining \eqref{cov} and \eqref{expest2}, we get
\begin{align*}
&\int_0^T\ell(x(t),u(t))\,dt\\
&\quad=\int_0^{T} \|u(t)+B^*P(T-t)x(t)\|^2dt+\<P(T)x_0,x_0\>+2\Re\<w,x(T)-x_0\>\\
&\quad=\int_0^{T} \|u(t)+B^*P(T-t)x(t)\|^2+2\Re\left\<u(t),B^*\hat{U}_{T}(T-t,0)w\right\>\\
&\qquad\qquad+2\Re\left\<B^*P(T-t)x(t),B^*\hat{U}_{T}(T-t,0)w\right\>dt+M_0\\
&\quad=\int_0^{T}\|u(t)+B^*P(T-t)x(t)+B^*\hat{U}_{T}(T-t,0)w\|^2dt+M
\end{align*}
where $M_0,M\in \R$ are constants independent of $u$ given by
\begin{align*}
M_0&:=\<P(T)x_0,x_0\>-2\Re\<w,x_0\>+2\Re\left\<x_0,\hat{U}_{T}(T,0)w\right\>,\\
M:&=M_0-\int_0^T\|B^*\hat{U}_{T}(T-t,0)w\|^2dt.
\end{align*}
This proves that, if \eqref{feedbacku} can give a solution to problem \eqref{sy} in mild sense, then this solution is optimal. 
Moreover, the closed loop trajectory (if it exists) given by the feedback law \eqref{feedbacku} must satisfy
$$
\dot{x}^*(t)=(A-BB^*P(T-t))x^*(t)-BB^*\hat{U}_T(T-t,0)w,\,x^*(0)=x_0\in\calH,\,t\in[0,T].
$$
By Proposition \ref{evolution}, this problem admits a unique solution $x^*\in L^2([0,T],\calH)$. Thus, \eqref{feedbacku} provides the optimal solution of \eqref{sy}.

For the general case, by Lemma~\ref{changestate}, the optimal optimal trajectory $x^*$ satisfies
$$
\frac{d(x^*-x_e)}{dt}(t)=(A-B(K^*K)^{-1}B^*P(T-t))(x^*(t)-x_e)-B(K^*K)^{-1}B^*\hat{U}_T(T-t,0)w.
$$
Simple calculation shows that the optimal control $u^*$ is given by
\begin{equation*}
u^*(t)=-(K^*K)^{-1}B^*P(T-t)(x^*(t)-x_e)-(K^*K)^{-1}B^*\hat{U}_{T}(T-t,0)w+u_e.
\end{equation*}  
\end{proof}

\subsection{ Proof of Theorem~\ref{maininf}}

In this subsection, we prove our main result.
\begin{proof}
Without loss of generality, we assume that the unique optimal steady state of problem~\eqref{OSSP} is at $(0,0)$ and $K=I$. 
{Fix $\calN$ to be some bounded neighborhood of $x_e=0$ and $T>0$.} Now let $(x^*,u^*)$ denote the optimal pair of problem \eqref{OCP} corresponding to the initial condition $x_0\in \calH$. {Let 
$$
\wt{u}^*(t):=u^*(t)+B^*P(T-t)x^*(t)=-B^*\hat{U}_{T}(T-t,0)w\,.
$$
By definition {of $\wt{u}^*$}, $\wt{u}^*$ is the optimal control of our modified system \eqref{refo}.} Recall from \eqref{cov} that
\begin{equation}\label{diff1}
\int_0^T\ell(x^*(t),u^*(t))\,dt=\int_0^{T}\|\wt{u}^*(t)\|^2dt+\<P(T)x_0,x_0\>+2\Re\<w,x^*(T)-x_0\>.
\end{equation}
{For an arbitrary $T_0\in [0,T]$, consider the modified input function defined by 
$$
\wt{u}(t):=
\left\{
\begin{alignedat}{2}
&0\qquad\qquad t\in[0,T_0)\\
&\wt{u}^*(t)\,\,\qquad t\in[T_0,T].
\end{alignedat}
\right.
$$
Denote the solution of \eqref{refo} corresponding to the modified input $\wt{u}(t)$ by $x(t)$, {then we can recover the input $u$ of the original system \eqref{sy} which generates the same trajectory $x$} by setting
$$
u(t):=
\left\{
\begin{alignedat}{2}
&-B^*P(T-t)x(t)\qquad\qquad\qquad\qquad\qquad t\in[0,T_0)\\
&-B^*P(T-t)x(t)-B^*\hat{U}_{T}(T-t,0)w\quad\; t\in[T_0,T].
\end{alignedat}
\right.
$$}
Now the cost functional is given by
\begin{equation}\label{diff2}
\int_0^T\ell(x(t),u(t))\,dt=\int_{T_0}^T\|\wt{u}^*(t)\|^2dt+\<P(T)x_0,x_0\>+2\Re\<w,x(T)-x_0\>\, ,
\end{equation}
where $x(T)=x^*(T)-U_T(T,T_0)\Phi^T_{T_0}(\wt{u}^*)$. 
In fact, by the standard properties of the evolution system \eqref{refo}, 
we have that
$$
x^*(T_0)-x(T_0)=\int_0^{T_0}U_T(T_0,s)B\wt{u}^*(s)ds=\Phi^T_{T_0}(\wt{u}^*)\,,
$$
and thus
$$
x^*(T)-x(T)=U_T(T,T_0)(x^*(T_0)-x(T_0))=U_T(T,T_0)\Phi^T_{T_0}(\wt{u}^*).
$$
We now consider the difference between \eqref{diff2} and \eqref{diff1}, and by Lemma \ref{exprate} and Lemma \ref{unibound}, we deduce the existence of some constant $M>0$ such that \begin{align}\label{wtubound}
\begin{split}
\|\wt{u}^{*}\|^2_{L^2([0,T_0],\mathcal{U})}:&=\int_0^{T_0}\|\wt{u}^*(t)\|^2dt \leq 2\Re\left\<w,-U_T(T,T_0)\Phi^T_{T_0}(\wt{u}^*)\right\>\\
&\leq Me^{-k(T-T_0)}\|\wt{u}^*\|_{L^2([0,T_0],\mathcal{U})}.
\end{split}
\end{align}
That is, $\|\wt{u}^{*}\|_{L^2([0,T_0],\mathcal{U})}\leq Me^{-k(T-T_0)}$. 
Since $x^*(T_0)=U_T(T_0,0)x^*(0)+\Phi^T_{T_0}(\wt{u}^*)$, we now have that, for some $M>0$, 
\begin{equation}\label{expxstar}
\|x^*(T_0)\|\leq M(e^{-kT_0}\|x_0\|+e^{-k(T-T_0)}).
\end{equation}
Recall from Lemma \ref{explicitoc} that $u^*(t)=-B^*P(T-t)x^*(t)-B^*\hat{U}_{T}(T-t,0)w$. We now prove that also~$u^*$ satisfies an estimate like \eqref{expxstar}. Indeed, since $P$ is continuous in strong topology and $P$ converges to $P_{\min}$ strongly as $t\to \infty$, for any $x\in \calH$ we have that
$$
\lim_{t\to0}P(t)x=0,\;\text{and}\;\lim_{t\to\infty}P(t)x=P_{\min}x.
$$
By the uniform boundedness principle, there exists some constant $M>0$ such that
\begin{equation}\label{lastest1}
\|P(t)\|<M,\; \forall t\in[0,\infty)\;.
\end{equation}

On the other hand, we claim that, for any possible values of $\tau,t,T$ with $0\leq\tau\leq t\leq T$, the operator $\hat{U}_T$ satisfies the estimate
\begin{align}\label{finest1}
\|\hat{U}_T(t,\tau)\|\leq Me^{-k(t-\tau)}
\end{align}
with some constants $M,k>0$. The proof will follow similar lines as the one of Lemma \ref{exprate}.

Recall that the evolution system corresponding to $\hat{U}_T(\cdot,\tau)$ is
\begin{align*}
\dot{z}(t)=(A^*-P(t)BB^*)z(t),\;z(\tau)=x_0,\;\tau\leq t\leq T.
\end{align*}
Let $y(t):=e^{\lambda (t-\tau)}z(t)$ with some sufficiently small $\lambda>0$ so that $A-BB^*P_{\min}+\lambda I$ generates an exponentially stable semigroup $\calT^{cl}$ with growth bound $-\beta<0$.

Then $A^*-P_{\min}BB^*+\lambda I$ generates the adjoint semigroup $(\calT^{cl})^*$ whose growth bound is also $-\beta$. Similar to \eqref{estony} and \eqref{Gronmat}, we have that
\begin{align}\label{hatUevo}
\begin{split}
&y(t)=(\calT^{cl}_{t-\tau})^*y(\tau)+\int_{\tau}^t(\calT^{cl}_{t-s})^*(P_{\min}-P(s))BB^*y(s)ds\\
&\|y(t)\|\leq M_1\|x_0\|+M_2e^{-\mu \tau}\int_{\tau}^t\|y(s)\|ds\;,
\end{split}
\end{align}
with some constants $M_1,M_2>0$. 
Now fix some $S>0$ such that $M_2e^{-\mu S}<\lambda$. In a similar manner as in the proof of Lemma \ref{exprate}, we can easily prove that if $\tau\geq S$, then \eqref{finest1} holds with some suitable coefficients $M,k$ independent of $T$.

For the case $\tau<S$ and $t\leq S$, observe from \eqref{hatUevo} that
\begin{align*}
\|y(t)\|\leq M_1\|x_0\|+M_2\int_\tau^t\|y(s)\|ds
\end{align*}
holds for some constants $M_1, M_2>0$. By Grönwall's inequality, we have
\begin{align}\label{secondcase}
\|\hat{U}_T(t,\tau)x_0\|\leq\|y(t)\|\leq M_1\|x_0\|e^{M_2(t-\tau)}\leq M_1\|x_0\|e^{-(t-\tau)}e^{(M_2+1)S}.
\end{align}
This implies that \eqref{finest1} holds with some suitable coefficients $M,k$.

Finally, if $t> S$, we have
\begin{align}\label{parstep2}
\begin{split}
\|\hat{U}_T(t,\tau)\|&=\|\hat{U}_{T}(t,S)\hat{U}_{T}(S,\tau)\|\\
&\leq Me^{-k(t-S)}\|\hat{U}_{T}(S,\tau)\| =Me^{-k(t-\tau)}e^{kS}\|\hat{U}_{T}(S,\tau)\|
\end{split}
\end{align}
with some $M,k>0$, since $\hat{U}_T(t,\tau)$ satisfies the estimate of the first case. By choosing $t=S$ in~\eqref{secondcase}, we get that
$$
\|\hat{U}_T(S,\tau)x_0\|\leq M_1\|x_0\|e^{M_2(S-\tau)}\leq M_1\|x_0\|e^{M_2S}.
$$
That is, $\hat{U}_T(S,\tau)$ is uniformly bounded in norm by some constant independent of $\tau$ and $T$. Combining this with \eqref{parstep2}, we prove the estimate \eqref{finest1} holds also for the case $\tau<S$ and $t>S$.

Finally, since $u^*(t)=-B^*P(T-t)x^*(t)-B^*\hat{U}_{T}(T-t,0)w$, combining \eqref{expxstar}-\eqref{finest1} gives
\begin{align}\label{expustar}
\|u^*(t)\|\leq M_1(e^{-kt}\|x_0\|+e^{-k(T-t)})+M_2e^{-k(T-t)}
\end{align}
for some sufficiently large constants $M_1,M_2>0$. Since $x_0\in\calN$ and $\calN$ is a bounded neighborhood of $x_e=0$, \eqref{expxstar} and \eqref{expustar} imply our main theorem.
\end{proof}

\section{Applications}\label{appli}
{In this section, we apply our main result to two examples.}

\begin{example} (The heat equation)
Let $\calH=L^2(0,1)$ endowed with the $L^2$ inner product. We define $A:D(A)\to \calH$ by
\begin{align*}
&D(A)=H^2(0,1)\cap H_0^1(0,1)\,,\\
&\text{and}\quad Ah=\frac{d^2 h}{d x^2},\quad\forall h\in D(A)\,,
\end{align*}
where
\begin{align*}
&H_0^1(0,1):=\Big\{f\in L^2[0,1]\;\Big|\;\frac{df}{dx}\in L^2[0,1],\ f(0)=f(1)=0\Big\},\\
&\text{and}\quad H^2(0,1):=\Big\{f\in L^2[0,1]\;\Big|\;\frac{df}{dx},\frac{d^2f}{dx^2}\in L^2[0,1]\Big\}.
\end{align*}
Then $A$ generates a $C_0$-semigroup $\calT$ on $\calH$. See, e.g., \cite[Example 2.6.8]{TusWei}.
Let  $\calU=\calY=\C$. We define $B:\calU\to \calH$ and $C:\calH\to \calY$ by
\begin{align*}
Bu=b&(x)u,\quad Ch(x)=\<h(x),c(x)\>,\\
\text{where}\quad &b(x)=\frac{1}{2\eps}\Id_{[x_0-\eps,x_0+\eps]}(x),\\
c(x)&=\frac{1}{2\upsilon}\Id_{[x_1-\upsilon,x_1+\upsilon]}(x)
\end{align*}
where $\Id_{S}$ is the characteristic function of a set $S$ and $v,\eps>0$ are sufficiently small constants. Clearly, $B$ and $C$ are bounded operators.

This functional framework is suitable to describe the following physical system: Consider a metal slab of length one with no heat flow at each end. Suppose we can act on the system by means of a point actuator around a point $x_0\in (0,1)$ to heat or cool this bar, and we can use a point sensor around a point $x_1\in(0,1)$ to measure the temperature. This model is described by the PDE
\begin{equation*}
\left\{
\begin{alignedat}{4}
&\frac{\ptl h}{\ptl t}(x,t)=\frac{\ptl^2 h}{\ptl x^2}(x,t)+\frac{1}{2\eps}\Id_{[x_0-\eps,x_0+\eps]}(x)u(t),\\
&h(x,0)=h_0(x),\\
&{h(0,t)=h(1,t)=0,}\\
&y(t)=\frac{1}{2\upsilon}\int_{x_1-\upsilon}^{x_1+\upsilon}h(x,t)dx.
\end{alignedat}
\right.
\end{equation*}

Consider the  problem of minimizing following running cost
\begin{equation*}
\ell(h,u):=\|Ch-y_0\|^2 + \|u\|^2
\end{equation*}
with $y_0\in\R$. That is, our goal is to constantly keep the temperature around $x_1$ close to $y_0$, but we can only influence the temperature of this slab through the part close to point $x_0$. 

By \cite[Lemma 2.4.7 and Theorem 8.1.6]{CurtainZwart}, $(A,B)$ is exponentially stabilizable and $(A,C)$ is exponentially detectable. We can thus apply Theorem \ref{maininf}, and conclude that the optimal control problem~\eqref{OCP} with running cost $\ell$ satisfies the exponential turnpike property at the optimal steady state given by Remark \ref{exanduniofossp}. Specifically, given a bounded set $\calN\in L^2(0,1)$, the optimal trajectory $h^*$ corresponding to initial condition $h_0\in \calN$ satisfies
$$
\|Ch^*(\cdot,t)-y_0\|\leq M(e^{-kt}+e^{-k(T-t)})
$$
for some constants $M,K>0$ for any $t\in [0,T]$. This implies that the average temperature on $[x_1-\upsilon,x_1+\upsilon]$ will be close to $y_0$ for most of the time. Similarly, the optimal control $u^*$ will stay close to a specific value for most of the time horizon. 
The setting of this example can be generalized to the multidimensional case as well, again relying on~\cite[Theorem 8.1.6]{CurtainZwart}.

\end{example}
\begin{example}(Elastic string)
Let $\calH=H_0^1(0,1)\times L^2[0,1]$, then $\calH$ is a Hilbert space with respect to the inner product
\begin{align}\label{innerproduct}
\left\<\bvec{f_1}{g_1},\bvec{f_2}{g_2}\right\>:&=\int_0^{1}\frac{df_1}{dx}(s)\ol{\frac{df_2}{dx}}(s)+g_1(s)\ol{g_2(s)}ds.\\
\end{align}
Let $A:D(A)\to\calH$ be defined on $D(A)=[H^2(0,1)\cap H_0^1(0,1)]\times H_0^1(0,1)$ by
\[A\bvec{f}{g}:=\bvec{g}{\frac{d^2f}{dx^2}},\quad\forall \bvec{f}{g}\in D(A)\; .
\]
Then $A$ generates a $C_0$-semigroup $\calT$ on $\calH$. It is well-known that this semigroup is a diagonalizable semigroup with pure imaginary eigenvalues. See \cite[Example~2.7.13]{TusWei} for reference.

Denote $\calU=L^2(0,1)$, take $\xi,\eta\in[0,1]$ with $\xi<\eta$ and consider the control operator $B\in L(\calU,\calH)$ defined by
\begin{align*}
Bu:=\bvec{0}{\Id_{[\xi,\eta]}u},\quad\forall u\in \calU.
\end{align*}
It is well-known that the pair $(A,B)$ defined above is exactly controllable, thus exponentially stabilizable. See \cite[Example 11.2.2]{TusWei}.

The previous functional framework is well-suited to describe the locally distributed control of a string equation with Dirichlet boundary conditions
\begin{equation*}
\left\{
\begin{alignedat}{4}
&\frac{\ptl^2 f}{\ptl t^2}(x,t)=\frac{\ptl^2 f}{\ptl x^2}(x,t)+\Id_{[\xi,\eta]}u(x,t),\\
&f(x,0)=f_0(x), \;\frac{df}{dt}(x,0)=g_0(x),\;\;\bvec{f_0}{g_0}\in D(A),\\
&f(0,t)=f(1,t)=0
\end{alignedat}
\right.
\end{equation*}
as a first-order linear control system on $\calH$ of the form~\eqref{sy}. 
We now define the observation operator $C:\calH\to L^2(0,1)$ by
\begin{align*}
C\bvec{f}{g}&:=g,\qquad\forall \bvec{f}{g}\in\calH
\end{align*}
and introduce the running cost $\ell$ by
\begin{align*}
\ell\left(\bvec{f}{g},u\right):=\left\|C\bvec{f}{g}\right\|^2&+\|u-u_0\|^2,\quad\forall \bvec{f}{g}\in\calH,
\end{align*}
for some possibly nonzero $u_0\in L^2(0,1)$. 
This choice of $\ell$ entails that we want to stabilize the string, while keeping the control as close as possible to the target value $u_0$.

In other words, given a time horizon $T>0$, we want to minimize the cost functional
\begin{align*}
J^T(f,u):= \int_0^T\int_0^{1}\left(\frac{\ptl f}{\ptl t}(x,t)\right)^2+(u(x,t)-u_0(x))^2dxdt.
\end{align*}
Now, \cite[Theorem~7.4.1]{TusWei} ensures that the pair $(A,C)$ is exactly observable, thus it is exponentially detectable. Thanks to Theorem \ref{maininf}, we deduce that the optimal control problem \eqref{OCP} satisfies the exponential turnpike property at the optimal steady state given by Remark \ref{exanduniofossp}. {In particular, if we let $f^*$ and $u^*$ denote the optimal trajectory and optimal control minimizing $J^T$, then the turnpike property implies that $\frac{df}{dt}(\cdot,t)$ will be close to $0$ in the $L^2$ sense for most of the time horizon, and $f^*(\cdot,t)$ and $u^*(\cdot,t)$ will stay close to the optimal steady state predicted by Remark \ref{exanduniofossp} with respect to the inner product \eqref{innerproduct} and $L^2$ inner product, respectively. This example can be also generalized to the multidimensional case. See \cite[Theorem~7.4.1]{TusWei}.}
\end{example}

\section{Conclusions}
In this paper we derive a sufficient condition of the exponential (integral) turnpike property in the context of infinite-dimensional generalized linear-quadratic optimal control problems. This condition extends the previous results~\cite{GruG21} in a finite dimensional setting and reveals the connection between the structural properties of the control system and the quantitative turnpike behavior of the optimally controlled system over a finite time horizon. Our proof is primarily based on the analysis of the exponential convergence of solutions to the differential Riccati equations to the algebraic counterpart, as first considered in~\cite{Zua}, and on a novel closed range test for stabilizability.

\bibliographystyle{abbrv}
\bibliography{references.bib}

\appendix
\section{An extension to Riesz's representation theorem}\label{appenA}
In the proof of Lemma \ref{test}, we claim $f(x)=2\Re\<w,x\>$ for some $w\in \calH$. We now give a short proof of this.
\begin{thm}\label{Riesz}
Suppose $\calH$ is a complex Hilbert space and $f:\calH\sra\R$ is a bounded linear functional over $\R$, then $f(x)=2\Re\<x,w\>$ for some $w\in \calH$.
\end{thm}
\begin{proof}
It's easy to verify $g(x):=f(x)-f(ix)i$ is a bounded linear functional over $\C$. By Riesz's representation theorem, we have $g(x)=\<x,2w\>$ for some $w\in\calH$. So, 
\begin{align*}
f(x)=\frac{g(x)+\ol{g(x)}}{2}=2\Re\<x,w\>.
\end{align*}
\end{proof}
\section{A counterexample to the existence of the adjoint steady state}\label{appenB}
In this section, we provide a counterexample showing that, in general, the existence of a unique minimizer of the optimal steady state problem does not guarantee the existence of the adjoint steady state. 
More precisely, we claim that the existence of a unique minimizer $(x_e,u_e)$ for the optimal steady state problem~\eqref{OSSP} does not ensure the existence of an element $w\in \calH$ such that
\begin{equation}\label{Lagu}
u_e=(K^*K)^{-1}B^*w-(K^*K)^{-1}v
\end{equation}
and
\begin{equation}\label{Lagran}
\begin{split}
\left\{
\begin{alignedat}{4}
&Ax_e+B(K^*K)^{-1}B^*w-B(K^*K)^{-1}v=0,\\
&C^*Cx_e+z-A^*w=0.
\end{alignedat}
\right.
\end{split}
\end{equation}

{Notice that \eqref{Lagu} and \eqref{Lagran} are obtained from the infinite dimensional analogy of the classic Lagrange multiplier approach, which has been utilized in \cite{TreZua}. Indeed, the vector $w$ corresponds to the optimal adjoint steady state in \cite{TreZua}. 
Our counterexample shows that the Lagrange multiplier approach fails in the general setting considered in this paper.}

Let $\calH=\calU=\calY=\ell^{\infty}$, $A=0$ and $C=K=I$. Let $B$ be defined by
$$
B(x_1,x_2,x_3,...):=(x_1,\frac{x_2}{2},\frac{x_3}{3},...)\, ,
$$
and let $z=0$ and $v=(1,\frac{1}{\sqrt{2}},\frac{1}{\sqrt{3}},...)$. {Then, the optimal steady state problem \eqref{OSSP} admits a unique minimizer $(x_e,u_e)=(0,0)$. {In fact, since $A = 0$, we must have $u_e=0$ while $x_e$ can be arbitrarily chosen. Since $C=I$ and $z=0$, $x_e=0$ minimizes running cost \eqref{cost}.}

We now check the existence of $w=(w_1,w_2,w_3,...)$. The first line of \eqref{Lagran} now reads
\begin{equation}\label{CounterB}
0+BB^*w-Bv=0.
\end{equation}
Since $B$ is self-adjoint, \eqref{CounterB} is equivalent to
\begin{equation}
(\frac{1}{1^2}w_1,\frac{1}{2^2}w_2,\frac{1}{3^2}w_3,...)=(\frac{1}{1}\frac{1}{\sqrt{1}},\frac{1}{2}\frac{1}{\sqrt{2}},\frac{1}{3}\frac{1}{\sqrt{3}},..).
\end{equation}
So, $w=(\sqrt{1},\sqrt{2},\sqrt{3},...)$, which is not an element in $\calH$. In fact, since our $A$ is bounded, we have $\calH=\calH_{-1}^d$, so $w$ is not even an element in $\calH_{-1}^d$.

{As we can see, the above optimal steady state problem does admit a unique minimizer $(0,0)$, but there exists no corresponding optimal adjoint steady state $w$ characterized by \eqref{Lagu} and \eqref{Lagran}. This issue has been properly resolved in Section~\ref{closedrange}, where we relied on the stabilizability assumption to get a closed range condition.}
}
\end{ietbody}
\end{document}